\documentclass{birkjour}
 \newtheorem{thm}{Theorem}[section]
 \newtheorem{cor}[thm]{Corollary}
 \newtheorem{lem}[thm]{Lemma}
 
 \theoremstyle{definition}
 \newtheorem{defn}[thm]{Definition}
 \theoremstyle{remark}
 \newtheorem{rem}[thm]{Remark}
 \newtheorem{ex}{Example}
 \numberwithin{equation}{section}

\begin{document}

%
%
%
%
%
%
%
%
%

\title[On the Exponent of the Schur multiplier of a Pair ... ]
 {On the Exponent of the Schur multiplier of a Pair of Finite $p$-Groups}

\author[F. Mohammadzadeh]{Fahimeh Mohammadzadeh}
\address{Department of Mathematic\\
Faculty of Sciences\\
Payame Noor University\\
Iran}
\email{F.mohamadzade@gmail.com}

\author[A. Hokmabadi]{Azam Hokmabadi}
\address{Department of Mathematic\\
Faculty of Sciences\\
Payame Noor University\\
Iran}
\email{ahokmabadi@pnu.ac.ir}

\author[B. Mashayekhy]{Behrooz Mashayekhy}
\address{Department of Pure Mathematics\\
Center of Excellence in Analysis on Algebraic Structures\\
Ferdowsi University of Mashhad\\
P. O. Box 1159-91775\\
 Mashhad\\
 Iran}
\email{bmashf@um.ac.ir}



\date{}

\begin{abstract}
In this paper, we find an upper bound for the exponent of the Schur multiplier of a pair $(G,N)$ of finite $p$-groups, when $N$ admits a complement in $G$. As a consequence, we show that the exponent of the Schur multiplier of a pair $(G,N)$ divides $\exp(N)$ if $(G,N)$ is a pair of finite $p$-groups of class at most $p-1$. We also prove that if $N$ is powerfully embedded in $G$, then the exponent of the Schur multiplier of a pair $(G,N)$ divides $\exp(N)$.
\end{abstract}

\maketitle
\section{Introduction and Motivation}

In 1998,  Ellis \cite{el98} extended the theory of the Schur multiplier for a pair of groups. By a pair of groups $(G,N)$ we mean a group $G$ with a normal subgroup $N$ of $G$. The Schur multiplier of a pair $(G,N)$ of groups is a
functorial abelian group $M(G,N)$ whose principal feature is a
natural exact sequence
\begin{eqnarray*}
 H_3(G) &\rightarrow& H_3(G/N) \rightarrow M(G,N) \rightarrow M(G)
\rightarrow M(G/N)\\
&\rightarrow& N/[N,G]\rightarrow (G)^{ab} \rightarrow (G/N)^{ab} \rightarrow 0
\end{eqnarray*}
in which  $H_3(G)$ is the third homology group of $G$ with integer coefficients.
In particular, if $N=G$, then $M(G,G)$ is the usual Schur multiplier $M(G)$.

It has been a considerable question that when $\exp(M(G))$ divides $\exp(G)$, in which $\exp(G)$ denotes the exponent of $G$.
Macdonald and Wamsely (see \cite{bay})
constructed an example of a group of exponent 4, whereas its Schur
multiplier has exponent 8, hence the conjecture is not true in
general. In 1973, Jones \cite{jons} proved
that the exponent of the Schur multiplier of a finite $p$-group of
class $c \geq 2$ and exponent $p^{e}$ is at most $p^{e(c-1)}$ and hence $\exp(M(G))$ divides $\exp(G)$ when $G$ is a $p$-group of class 2. In 1987, Lubotzky and Mann \cite{lub} proved that $\exp(M(G))$ divides $\exp(G)$ when $G$ is a powerful $p$-group. A
result of Ellis \cite{el2001} shows that if $G$ is a $p$-group of class $k
\geq 2$ and exponent $p^{e}$, then $\exp(M(G))\leq p^{e\lceil k/2\rceil}$, where $\lceil k/2\rceil$ denotes the smallest integer
$n$ such that $n \geq k/2$. Moravec \cite{mor} showed that
$\lceil k/2\rceil$ can be replaced by $2\lfloor \log_2{k}\rfloor$
which is an improvement if $k\geq 11$. He \cite{mor} also proved that if $G$ is
a metabelian group of exponent $p$, then $\exp(M(G))$ divides $p$.
Kayvanfar and  Sanati \cite{kay} proved that if $G$ is a $p$-group, then $\exp(M(G))$ divides
$\exp(G)$ when $G$ is a finite $p$-group of class 3, 4 or 5 with some conditions.  The authors \cite{mhm}  extended the result and proved that $\exp(M(G))$ divides $\exp(G)$ when $G$ is a finite $p$-group of class at most $p-1$.\

On the other hand, Ellis \cite{el98} proved  that $\exp(M(G,N))$ divides $|N|$ for any pair $(G,N)$ of finite groups, in which $|N|$ denotes  the order of $N$. Now a question that can naturally arise, is whether $\exp(M(G,N))$ divides $\exp(N)$ when $N$ is a proper normal subgroup of $G$. In this paper, first we present an example to give
 a negative answer to the question. Second, we give some conditions
 under which the exponent of $M(G,N)$ divides the exponent of $N$.

 In Section 2, we give an upper bound for $\exp(M(G,N))$  in terms of $\exp(N)$, when $(G,N)$ is a pair of finite $p$-groups such that $N$ admits a complement in $G$, and apply it to  prove that if $(G,N)$ is a pair of finite $p$-groups of class at most $p-1$ (i.e $[N, \ _{p-1}G]=1$), then $\exp(M(G,N))$ divides $\exp(N)$.
Finally in Section 3, we show that if $(G,N)$ is a pair of finite $p$-groups and $N$ is powerfully embedded in $G$, then $\exp(M(G,N))$ divides $\exp(N)$.

\section{ Nilpotent pairs of $p$-groups}

I. D. Macdonald and J.W. Wamsley \cite{bay} gave an example which shows that $\exp(M(G,G))$ dose not divide $\exp(G)$, in general. The following example shows that $\exp(M(G,N))$ dose not divide $\exp(N)$ when $N$ is a proper normal subgroup of $G$. \\

\begin{ex}
Let $D=A >\!\!\!\! \lhd <x_1>$, where $A=<x_2>\times <x_3>\times<x_4>\times<x_5>\cong {\mathbf Z}_{{4}}\times{\mathbf Z}_{{4}}\times{\mathbf Z}_{{4}}\times{\mathbf Z}_{{2}}$ and $x_1$ is an automorphism of order 2 of $A$ acting in the following way:
$$[x_2,x_1]=x_2^2,\ \ [x_3,x_1]=x_3^2,\ \ [x_4,x_1]=x_4^2,\ \ [x_5,x_1]=1.$$
There exists an automorphism $a$ of $D$ of order 4 acting on $D$ as follows:
$$[x_1,a]=x_3,\ \ [x_2,a]=x_2^2 x_3^2 x_4^3\ \ [x_3,a]=x_5,\ \ [x_4,a]=x_2^2,\ \ [x_5,a]=x_3^2.$$
Form $N=D >\!\!\! \lhd <a>$ and put $G=N >\!\!\! \lhd <b>$, where $b^2=1$ and
$[x_1,b]=x_2,\  [x_2,b]=x_2^2x_4^3x_5,\  [x_3,b]=x_4,\  [x_4,b]=x_3^2x_4^2,\ [x_5,b]=x_2^2x_3^2x_4^2,\  [a,b]=x_1.$
 Moravec \cite{mor} showed that the group $G$ is a nilpotent group of class 6 and exponent 4 and $M(G)\cong{\mathbf Z}_{{2}}\times{\mathbf Z}_{{4}}\times{\mathbf Z}_{{8}}$.
 Ellis \cite{el98} proved that if $G=K >\!\!\!\! \lhd Q$, then $M(G)\cong M(G,K)\oplus M(Q)$. This implies that $M(G,N)\cong M(G)$. Therefore $\exp(M(G,N))=8$ dose not divide $\exp(N)=4$.
 \end{ex}

Here we first give an upper bound for the exponent of $M(G,N)$ in terms of the exponent of $N$, when $(G,N)$ is a pair of finite $p$-groups such that $N$ admits a complement in $G$. Since our proof relies on commutator calculations, we need to state the following lemmas.

\begin{lem} (\cite{st}). Let $x_1, x_2,
\ldots , \ x_r$ be any elements of a group and $\alpha$ be a nonnegative integer. Then
$$(x_1 x_2 ... x_r)^{\alpha}=x_{i_1}^{\alpha}x_{i_2}^{\alpha} ...x_{i_r}^{\alpha} \upsilon_1^
{f_1(\alpha)}\upsilon_2^{f_2(\alpha)}\cdots \ ,$$
where $\{i_1, \ i_2, \ldots,\ i_r \} = \{1, 2, \ldots, r \}$ and $\upsilon_1,\upsilon_2, \cdots$  are commutators of weight at least two in the letters $x_i^,$ in ascending order and
\begin{equation}
f_i(\alpha)=a_1{\alpha \choose 1}+ a_2{\alpha \choose 2}+\cdots+a_{w_i}{\alpha \choose w_i},
\end{equation}
with $a_1, \ldots, a_{wi} \in \mathbf{Z} $ and $w_i$ is the
weight of $\upsilon_i$ in elements $x_1, \ldots, x_r$.
\end{lem}

\begin{lem} (\cite{st}). Let $\alpha$
be a fixed integer and $G$ be a nilpotent group of class at most
$k$. If $b_1, \ldots, b_r \in G$ and $r<k$, then
$$[b_1,\ldots,b_{i-1},b_i^{\alpha},b_{i+1},\ldots,b_r]=[b_1,...,b_r]^{\alpha}
\upsilon_1^{f_1(\alpha)}\upsilon_2^{f_2(\alpha)}\cdots,$$ where
$\upsilon_1, \upsilon_2, \ldots$ are commutators in $b_1, \ldots,b_r$ of weight strictly
greater than $r$, and every $b_j$, $1\leq j \leq r$, appears in
each commutator $\upsilon_i$. The $f_i(\alpha)$ are of the form (2.1), with $a_1, \ldots, a_{w_i} \in \mathbf{Z}$ and $w_i$ is the weight of $\upsilon_i$ (in $b_1, \ldots, b_r$)
minus $(r-1)$.
\end{lem}

It is noted by Struik \cite{st} that the above lemma can be proved similarly if $[b_1, \dots ,b_{i-1},b_i^{\alpha},b_{i+1}, \dots ,b_r]$ and $[b_1, \dots ,b_r]$ are replaced by arbitrary commutators (that is monimial commutators with parentheses arranged arbitrarily).

To prove the main results we require the following notions.
\begin{defn}
 A relative central
extension of a pair $(G,N)$ of groups consists of a group homomorphism
$\sigma: M \rightarrow G$ together with an action of $G$ on $M$
such that \\
i) $\sigma (M)=N$;\\
ii) $\sigma (m^g)=g^{-1} \sigma (m)g$, for all $m \in M, g \in G$;\\
iii) $m^{\sigma(m_1)}=m_1^{-1} m m_1$, for all $m \in M, g \in G$;\\
iv) $G$ acts trivially on $ker \sigma$.
\end{defn}

Let $(G,N)$ be a pair of groups and $\sigma: M \rightarrow G$ be a relative central
extension of $(G,N)$. The $G$-commutator subgroup of $M$ is defined
the subgroup $[M,G]$ generated by all the $G$-commutators $[m,g]=m^{-1}m^g ,$
where $m^{g}$ is the action of $g$ on $m$, for all $g \in G, m \in M$. Also for all positive integer $n$, we define
$$Z_n(M,G)= \{ m \in M | \ [m,g_1,g_2,\ldots,g_n]=1 , \ for \ all\ g_1, g_2, \ldots , g_n \in G \},$$ in which $ [m,g_1,g_2, \ldots,g_n]$ denotes $[\cdots[[m,g_1],g_2], \ldots,g_n]$. It is easy to see that $Z_n(M,G)\leq Z_n(M)$. The subgroup $Z_1(M,G)$ is well known as the $G$-center of $M$ which is denoted by $Z(M,G)$.

Let $(G,N)$ be a pair of groups and $k$ be a positive integer. We define
$\gamma_{k+1}(N,G)= [N, \ _kG]$ in which $[N, \ _kG]=[ \cdots [[N,\underbrace{G],G], \ldots ,G]}_{k-times}$.
A pair $(G,N)$ of groups is called nilpotent of class $k$ if $\gamma_{k+1}(N,G)=1$ and $\gamma_{k}(N,G)\neq 1$. It is clear that any pair of finite $p$-groups is nilpotent.

\begin{defn}
 A relative central
extension $\sigma: N^*\rightarrow G$ of a pair $(G,N)$ is
called a covering pair if there exists a subgroup $A$ of $N^*$
such that \\
i) $A \leq Z(N^*,G)\cap [N^*,G]$;\\
ii) $A\cong M(G,N)$;\\
iii) $N\cong N^*/A$.\

Ellis proved that any pair of finite groups has at least one covering pair [2, Theorem 5.4].

\end{defn}

Hereafter in this section, we suppose that $(G,N)$ is a pair of finite groups and $K$ is the complement of $N$ in $G$. Also, suppose that
$\sigma: N^*\rightarrow G$ is a covering pair of $(G,N)$ with a subgroup of $A$ of $N^*$ such that $A \leq Z(N^*,G)\cap [N^*,G]$, $A\cong M(G,N)$ and $N\cong N^*/A$. Then for all $k \in K$, the homomorphism $\psi_k:N^* \rightarrow N^*$ defined by $ n^* \mapsto {n^*}^k$ is an automorphism of  $N^*$ in which ${n^*}^k$ is induced by the
action of $G$ on $N^*$. Considering the homomorphism $\psi:K \rightarrow Aut(N^*)$ given by $\psi(k)=\psi_k$ for all $k \in K$, we form the semidirect product of $N^*$
by $K$ and denote it by $G^*=N^*K$. Then it is easy to check
that the subgroups $[N^*,G]$ and $Z(N^*,G)$ are contained in
$[N^*,G^*]$ and $Z(N^*, G^*)$, respectively. If $\delta : G^*
\rightarrow G$ is the map given by
$\delta(n^*k)=\sigma(n^*)k,$ for all $n^* \in N^*$ and $k\in K$,
then $\delta$ is an epimorphism with $\ker \delta=\ker \sigma$.

\begin{lem}
 By the above notation, let  $(G,N)$ be
a nilpotent pair of finite groups of class $k$ and $\exp(N)=p^e$.
 Then every commutator of weight $w$ ($w \geq 2$) in $[N^*,\ _{w-1}G^*]$ has an order dividing $p^{e+m(k+1-w)}$, where $m= \lfloor \log_p k \rfloor $.
 \end{lem}

\begin{proof}
We use reverse induction on $w$ to prove the lemma.
Since $(G,N)$ is nilpotent of class $k$ and $N\cong N^*/A$, $G\cong G^*/A$ and $A\leq Z(N^*,G^*)$, we have $[N^*,\ _{k+1}G^*]=1$. On the other hand, $\exp(N)=p^e$ implies that $[{N^*}^{p^e},G^*]=1$. Hence the result follows for $w \geq k+1$ by Lemma 2.2.
 Now assume that $l<k+1$ and the result is true for all $w>l$. We will prove the result for $l$. Put
$\alpha=p^{e+m(k+1-l)}$ with $m= \lfloor \log_p k \rfloor $ and let $u=[n, x_2, \ldots , x_l]$ be a commutator of weight $l$, where $n\in N^*$ and $x_2, \ldots , x_l \in G^*$. Then by Lemma 2.2, we have
$$ [n^{\alpha}, x_2, \ldots , x_l]=[n, x_2, \ldots , x_l]^{\alpha}
\upsilon_1^{f_1(\alpha)} \upsilon_2^{f_2(\alpha)} \cdots,$$
where $\upsilon_i$ is a commutator on $n, x_2, \ldots, x_l$ of weight $w_i$ such that $l < w_i \leq k+1 $, and  $f_i(\alpha)=a_1{\alpha \choose 1}+
a_2{\alpha \choose 2}+ \dots +a_{k_i}{\alpha \choose k_i}$, where
$k_i=w_i - l + 1 \leq k$, for all $ i \geq 1$.
One can easily check that $p^t$ divides $p^{t+m} \choose s$ with $m=\lfloor \log_p k \rfloor$, for any prime $p$ and any positive integers $t,s$ with $s \leq k$.
This implies that $p^{e+m(k-l)}$ divides the $f_i(\alpha)$'s and so by induction hypothesis
$\upsilon_i^{f_i(\alpha)} =1$, for all $ i \geq 1$.
On the other hand, it is clear that $[n^{\alpha},x_2, \ldots, x_l] =1$. Therefore $u^{\alpha}=1$ and this
completes the proof.
\end{proof}

\begin{thm}
If  $(G,N)$ is
a nilpotent pair of finite groups of class $k$ and $N$ is a $p$-group of exponent
$p^e$, then $\exp([N^*,G^*])$ divides $p^{e+m(k-1)}$,
where $m= \lfloor \log_p k \rfloor $.
\end{thm}

\begin{proof}
Every element $g \in [N^*,G^*]$ can be expressed as $g=y_1 y_2 \cdots
y_n$, where $y_i=[n_i,g_i]$ for $ n_i \in N^*, g_i \in G^*$. Put $\alpha = p^{e+m(k-1)} $.
By Lemma 2.1, we have
$$g^{\alpha} = y_{i_1}^{\alpha} y_{i_2}^{\alpha} \cdots y_{i_n}^{\alpha}
\upsilon_1^{f_1(\alpha)} \upsilon_2^{f_2(\alpha)} \cdots,$$
where $\{i_1, i_2, \ldots,i_n \}=\{1, 2, \ldots, n \}$ and $\upsilon_i$ is a basic commutator of weight $w_i$ in $y_1, y_2, \ldots, y_n$, with $2\leq w_i \leq k$, for all $i \geq 1$, and also $f_i(\alpha)$ is of the form (2.1). Hence by an argument similar to the  proof of Lemma 2.5 $p^{e+m(k-2)}$
divides $f_i(\alpha)$. Then applying Lemma 2.5, we have $\upsilon_i^{f_i(\alpha)} =1$,
for all $i \geq 1$, and $y_j^{\alpha} =1$, for all $j$, $1\leq j \leq n$.
We therefore have $g^{\alpha}=1$ and the desired result follows.
\end{proof}

An upper bound  for the exponent of the Schur multiplier of some pairs of finite groups is given in the following theorem.
\begin{thm}
Let $(G,N)$ be a nilpotent pair of finite groups of class $k$ such that $\exp(N)=p^e$. Then $\exp(M(G,N))$ is a divisor of $p^{e+m(k-1)}$, where $m= \lfloor \log_p k \rfloor $.
\end{thm}

\begin{proof}
 The result follows by Theorem 2.6 and the fact that $M(G,N)\cong A \leq [N^*,G]\leq [N^*,G^*]$.
\end{proof}

The following corollary gives a condition under which the exponent of the Schur multiplier of a pair $(G,N)$ divides the exponent of $N$.
\begin{cor}
Let $(G,N)$ be a pair of finite $p$-groups of class at most $p-1$. Then $\exp(M(G,N))$ divides $\exp(N)$.
\end{cor}

\begin{rem}
Let $G$ be a finite $p$-group of class $k$ with $exp(G)=p^e$ . Since $M(G,G)=M(G)$, Theorem 2.7 implies that $exp(M(G))$ divides $p^{e+[\log_pk](k-1)}$. It is easy to see that this bound improves the bound $p^{(2e\lfloor \log_2{k}\rfloor)}$ given by Moravec \cite{mor}. For example for any $p$-group $G$ of class $k$, $2\leq k \leq p-1$ with $exp(G)=p^e$, we have $p^{e+[\log_pk](k-1)} \leq p^{(2e\lfloor \log_2{k}\rfloor)}$.
\end{rem}

\begin{rem}
Let $(G,N)$ be a pair of finite nilpotent groups of class at most $k$. Let $S_1,S_2, \ldots, S_n$ be all the Sylow subgroups of $G$.
By [2, Corollary 1.2], we have $$ M(G,N)=M(S_{1} , S_{1} \cap N) \times \dots
\times M(S_{n} , S_{n} \cap N).$$ Put  $m_i= \lfloor \log_{p_i} k \rfloor $, for all $i$,
 $1\leq i \leq n$. Then by Theorem 2.7, we have
$$\exp( M(G,N) \mid \prod^n_{i=1}p_i^{e_i+m_i(k-1)},$$ where
$p_i^{e_i} = \exp(S_{i})$.
\end{rem}

\section{  Pairs of powerful $p$-groups}

In 1987, A. Lubotzky and A. Mann \cite{lub} defined powerful $p$-groups which are used for studying $p$-groups.
They gave some bounds for the order, the exponent and the number of
generators of the Schur multiplier of a powerful $p$-group. Also, they showed that $\exp(M(G))$ divides $\exp(G)$ when $G$ is a powerful $p$-group. The purpose of this section is to show that
if $(G,N)$ is a pair of finite $p$-groups and $N$ is powerfully embedded in $G$, then the exponent of $M(G,N)$ divides the exponent of $N$.
 Throughout this section $\mho_i(G)$ denotes the subgroup of
$G$ generated by all $p^{i}$th powers of elements of $G$. It is easy to see that $\mho_{i+j}(G) \subseteq \mho_i(\mho_j(G))$, for all positive integers $i,j$.
\begin{defn}
$(i)$ A $p$-group $G$ is
called powerful if $p$ is odd and $G' \leq \mho_1(G)$, or
$p=2$ and $G' \leq \mho_2(G)$.\\
$(ii)$ Let $G$  is a $p$-group and $N\leq G$. Then $N$ is   powerfully
embedded in $G$ if $p$ is odd and $[N,G]\leq \mho_1(N)$, or $p=2$ and $[N,G]\leq \mho_2(N)$.
\end{defn}

Any powerfully embedded subgroup is itself a powerful
$p$-group and must be normal in the whole group. Also a $p$-group is
powerful exactly when it is powerfully embedded in itself. While it
is obvious that factor groups and direct products of powerful
$p$-groups are powerful, this property is not subgroup-inherited \cite{lub}.
The following lemma gives some properties of powerful $p$-groups.\\

\begin{lem}
 (\cite{lub}). The following statements
hold for a powerful $p$-group
$G$.\\
$(i)$ $\gamma_i(G), G^{i}, \mho_i(G), \Phi(G)$ are powerfully
embedded in $G$.\\
$(ii)$ $\mho_i(\mho_j(G))=\mho_{i+j}(G)$.\\
$(iii)$ Each element of $\mho_i(G)$ can be written as $a^{p^{i}},$
for
some $a\in G$ and hence $\mho_i(G)=\{g^{p^{i}}: g\in G\}$.\\
$(iv)$ If $G=\langle a_1,a_2,...,a_d\rangle$, then $\mho_i(G)=\langle
a_1^{p^{i}},a_2^{p^{i}},...,a_d^{p^{i}}\rangle$.
\end{lem}

\begin{lem}
 (\cite{lub}). Let $N$ be powerfully embedded in $G$. Then $\mho_i(N)$ is  powerfully embedded in $G$.
\end{lem}

The proof of the following lemma is straightforward.
\begin{lem}
Let $M$ and $G$ be two groups with an action of $G$ on $M$. Then for all $m,n \in M$, $g,h \in G$, and any integer $k$ we have the following equalities. \\
$(i)$ $[mn,g]=[m,g]^n[n,g]$;\\
$(ii)$ $[m,gh]=[m,h][m,g]^h$;\\
$(iii)$ $[m^{-1},g]^{-1}=[m,g]^{m^{-1}}$;\\
$(iv)$ $[m,g^{-1}]^{-1}=[m,g]^{g^{-1}}$;\\
$(v)$ $[m,g^{-1},h]^g[m,[g,h^{-1}]]^h[[m^{-1},h]^{-1},g]^m=1$;\\
$(vi)$ $[m^k,g]=[m,g]^k [m,g,m]^{k(k-1)/2}  \pmod{[M, \ _3G]}$.
\end{lem}

\begin{lem}
Let $(G,N)$ be a pair of finite $p$-groups and $\sigma : N^* \rightarrow G$ be a relative central extension of $(G,N)$. Suppose that $M$ and $K$ are two normal subgroups of $N^*$. Then $M \leq K$ if  $M \leq K[M,G]$.
 \end{lem}

\begin{proof}
 Applying Lemma 3.4 we have
$$ M \leq K[M,G]\leq K[K[M,G],G]\leq K[K,G][M,G,G]\leq \dots \leq K[M, \ _iG], $$
for all $i\geq 1$. On the other hand, since $G$ is a finite $p$-group, there exists an integer $l$ such that  $[N, \ _lG]=1 $. Hence $[N^*, \ _{l+1}\ G]=1$ and the result follows.
\end{proof}

\begin{lem}
 Let $(G,N)$ be a pair of finite $p$-groups and $\sigma : N^* \rightarrow G$ be a relative central extension of $(G,N)$. Let $M$ be a normal
subgroup of $H$. Then the following statements hold.\\
$(i)$ If $p>2$, then $[\mho_1(M),G]\subseteq \mho_1([M,G] )[M, \ _3G]$.\\
$(ii)$ If $p=2$, then $[\mho_2(M),G]\subseteq \mho_2([M,G] ) \mho_1([M, \ _2G])[M, \ _3G]$.
\end{lem}
\begin{proof} $(i)$ It is enough to show that $[m^p,g]\in \mho_1([M,G] )[M, \ _3G]$, for all $m \in M, g \in G$.
 By Lemma 3.4  $[m^p,g]={[m,g]}^p {[m,g,m]}^{p(p-1)/2}\ \ \ \pmod {[M, \ _3G]}.$ Since $p$ is odd and $p \mid \frac{p(p-1)}{2}$ we have ${[m,g]}^p {[m,g,m]}^{p(p-1)/2} \in \mho_1([M,G] )$. Now the result holds.\\
 $(ii)$ The proof is similar to $(i)$.
\end{proof}

\begin{lem}
Let $(G,N)$ be a pair of finite $p$-groups and $\sigma : N^* \rightarrow G$ be a relative central extension of $(G,N)$.  Suppose that $K \leq N^*$. Then the following statements hold.\\
$(i)$ If $p>2$, then $[K,G] \leq \mho_1 (K)$ if and only if $[K/[K,\  _2G], G]\leq \mho_1( K/[K,\  _2G])$.\\
$(ii)$ If $p=2$, then $[K,G] \leq \mho_2 (K)$ if and only if $[K/[K,\  _2G], G]\leq \mho_2( K/[K,\  _2G])$.\\
$(iii)$ If $p=2$, then $[K,G] \leq \mho_2 (K)$ if and only if $[K/\mho_1([K, G]), G]\leq \mho_2( K/\mho_1([K, G]))$.
\end{lem}
\begin{proof} $(i)$ Let $[K,G] \leq \mho_1 (K)$ and put $H=[K, \ _2G]$. Then
$$ [\frac{K}{H},G]=\frac{[K,G]H}{H} \leq \frac{\mho_1(K)H}{H}=\mho_1(\frac{K}{H}),$$ as desired. Sufficiency follows by Lemma 3.5.\\
$(ii)$ The proof is similar to $(i)$.\\
$(iii)$  Necessity follows as for (i). Let $[K/\mho_1([K, G]), G]\leq \mho_2( K/\mho_1([K, G]))$. Then $[K,G]\leq \mho_2(K) \mho_1([K,G])$. On the other hand, $\mho_1([K/\mho_1([K, G]), G])$. Thus $[K/\mho_1([K, G]), G]$ is abelian and so $\Phi([K/\mho_1([K, G]), G])=1$. This implies that $\Phi([K,G])=\mho_1([K,G])$. Therefore $[K,G]\leq \mho_2(K)$.
\end{proof}

The following useful remark is a consequence of Lemma 3.7.
\begin{rem}
Let $(G,N)$ be a pair of finite $p$-groups and $\sigma : N^* \rightarrow G$ be a relative central extension of $(G,N)$. Let $K \leq N^*$. Then to prove that $[K,G] \leq \mho_1 (K)$ ( $[K,G] \leq \mho_2 (K)$ for $p=2$)  we can assume that\\
$(i)$ $[K, \ _2G]=1$;\\
$(ii)$ $\mho_1(K)=1\ (  \mho_2(K)=1$ for $p=2$ ) and try to show that $[K,G]=1$;\\
$(iii)$ $\mho_1([K,G])=1$ whenever $p=2$.
\end{rem}
\begin{lem}
Let $(G,N)$ be a pair of finite $p$-groups and $\sigma : N^* \rightarrow G$ be a covering pair of $(G,N)$. Let $N$ be powerfully embedded in $G$.  \\
$(i)$ If $p>2$, then $[\mho_n([N^*,G]),G]\leq \mho_1(\mho_n([N^*,G]))$ .\\
$(ii)$ If $p=2$, then $[\mho_n([N^*,G]),G]\leq \mho_2(\mho_n([N^*,G]))$.
\end{lem}
\begin{proof}
 $N^*$ has a subgroup $A$ such that $A \leq Z(N^*,G)\cap [N^*,G]$, $A\cong M(G,N)$ and $N\cong N^*/A$.\\
$(i)$ Let $p>2$. We use induction on $n$. If $n=0$, then by Remark 3.8 we may assume that $[[N^*,G],G,G]=1$, $\mho_1([N^*,G])=1$ and
 we should show that $[[N^*,G],G]=1$.
  Since $N$ is powerfully embedded in $G$, we have  $[N,G]\leq  \mho_1(N)$, and therefore  $[N^*,G]\leq  \mho_1(N^*)A$. Now we claim that  $\mho_1(N^*)\leq Z(N^*,G)$.
  To prove the claim, let $ a \in N^*$ and $b\in G$. Since $\gamma_3(\langle a, [N^*,G]\rangle)=1$, we have  $cl(<a,a^b>) \leq cl(<a, [N^*,G]>)\leq 2$  ($cl(H)$ denotes the nilpotency class of $H$). On the other hand, Lemma 3.4 implies that $$
  (a^p)^b=a^p[a^p,b]\equiv  a^p[a,b]^p[a,b,a]^{p(p-1)/2} \pmod{[<a>, \ _3G]}.$$  Therefore $(a^p)^b=a^p$ since $[[N^*,G],G,G]=1$ and $\mho_1([N^*,G])=1$.
  Hence  $\mho_1(N^*)\leq Z(N^*,G)$ as desired. Thus $[N^*,G] \leq  \mho_1(N^*)A\leq Z(N^*,G)$ and the result follows for $n=0$.
  
  Now suppose that the induction hypothesis is true for $n=k$. The first step of induction implies that $[N^*,G]$ is powerful.
  Using Lemmas 3.5 and 3.6, one can see that if $H$ is a subgroup of $N^*$ and $[H,G] \leq  \mho_1(H)$, then $[\mho_1(H),G] \leq \mho_1(\mho_1(H))$.  Hence by Lemma 3.2 and induction hypothesis we have
$$[\mho_{k+1}([N^*,G]),G]= [\mho_1(\mho_{k}([N^*,G])),G]\leq \mho_1(\mho_1(\mho_{k}([N^*,G])))$$ $$=\mho_1(\mho_{k+1}([N^*,G]))$$ which completes the proof.\\
$(ii)$ Let $p=2$. The proof is similar to (i), but we need to prove that if $H$ is a subgroup of $N^*$ and $[H,G]\leq \mho_2(H)$, then $[ \mho_1(H),G] \leq \mho_2(\mho_1(H))$. By  Remark 3.8,
  for $a \in H, b \in G$ we have $[a^4,b]=[a^2,b]^2=1$. So $a^4 \in Z(H,G)$ and $\mho_2(H)\leq Z(H,G)$. Then $[H,G]\leq \mho_2(H)\leq Z(H,G)$. Therefore $[a^2,b]=[a,b]^2$ and
  \begin{equation}
  [\mho_1( H),G]=\mho_1([H,G]).
 \end{equation}
 On the other hand, since $\mho_2(H)\leq Z(H,G)$, we have
 $$\mho_1(\mho_2(H))=<({a_1}^4 \dots {a_k}^4)^2| a_i \in H>=<{a_1}^8 \dots {a_k}^8>=\mho_3(H)\leq \mho_2(\mho_1(H)). $$
 Hence (3.1) implies that $[\mho_1(H),G] \leq \mho_2(\mho_1(H))$ which completes the proof of the above claim.
\end{proof}
\begin{lem}
Let $H$ and $G$ be two arbitrary groups with an action of $G$ on $N$. If $x \in H$ and $g\in G$, then $$[x^n,g]=[x,g]^n c, $$ where $M=\langle x,[x,g] \rangle$ and $ c \in \gamma_2(M)$.
\end{lem}
\begin{proof} Applying Lemma 2.1, we have
$$  [x^n,g]=(x^n)^{-1}(x^n)^{g}=(x)^{-n}(x^g)^n=(x)^{-n}(x[x,g])^n=[x,g]^n c, $$
where $M=\langle x,[x,g] \rangle, c \in \gamma_2(M)$.
\end{proof}

Now we can state the main result of this section.
\begin{thm}
  Let $(G,N)$ be a pair of finite $p$-groups in which $N$ is powerfully embedded in $G$. Then
  $\exp(M(G,N))$ divides $ \exp(N)$.
  \end{thm}
\begin{proof} Let $p>2$ and $\sigma : N^* \rightarrow G$ be a covering pair of $(G,N)$ with a subgroup $A$ such that $A \leq Z(N^*,G)\cap [N^*,G]$, $A\cong M(G,N)$ and $N\cong N^*/A$. It is enough to show that $\exp([N^*,G])=\exp(N^*/Z(N^*,G))$. For this we use induction on $k$ and show that
\begin{equation}
\mho_k([N^*,G])=[\mho_k(N^*),G].
\end{equation}
If $k=0$, then (3.2) holds. Now assume that (3.2) holds, for  $k=n$.
Working in powerful $p$-group  $N^*/A$ we get  $\mho_{n+1}(N^*/A)=\mho_1 (\mho_n(N^*/A))$ by Lemma 3.2. Hence
\begin{equation}
 \frac{\mho_{n+1}(N^*)A}{A}=\mho_1(\frac{
\mho_n(N^*)A}{A})=\frac{\mho_1(\mho_n(N^*)A)A}{A}.
\end{equation}
Then Lemmas 3.6 and 3.9 and induction hypothesis imply that
\begin{eqnarray*}
 [\mho_{n+1}(N^*),G]=
[\mho_1 (\mho_n(N^*)A)A,G] & \leq &
\mho_1 ([\mho_n(N^*)A,G])[\mho_n(N^*)A, \ _3G] \\& \leq & \mho_1 ([\mho_n(N^*),G])[\mho_n(N^*)A, \ _2G] \\&\leq& \mho_1 (\mho_n([N^*,G])[\mho_n([N^*,G]),G] \\&\leq& \mho_1 (\mho_n([N^*,G])=\mho_{n+1}([N^*,G]).
\end{eqnarray*}
For the reverse inclusion, we show that
$$\mho_{n+1}([N^*,G]) \equiv 1 \pmod{[\mho_{n+1}(N^*),G]}.$$
Since by $(3.4)$, $[\mho_{n+1}(N^*),G])=[\mho_1 (\mho_n(N^*)A)A,G]=[\mho_1 (\mho_n(N^*)A),G]$, it follows that
$\mho_1 (\mho_n(N^*)A)A \leq Z(N^*,G)  \pmod{[\mho_{n+1}(N^*),G]}$.\\
On the other hand, since $N$ is powerfully embedded in $G$, we have
\begin{eqnarray*}
 [\mho_{n}(N^*),G]=[\mho_{n}(N^*)A,G] &\leq& \mho_1 (\mho_n(N^*)A)A \\&\leq& Z(N^*,G) \pmod{[\mho_{n+1}(N^*),G]}.
 \end{eqnarray*}
Therefore $[\mho_{n}(N^*),G,G]\equiv 1 \pmod{[\mho_{n+1}(N^*),G]}.$\\
 Moreover, by Lemma 3.10 we have
 $$[\mho_1 (\mho_n(N^*)A),G][\mho_{n}(N^*),G,G]=\mho_1( [ \mho_n(N^*),G])[\mho_{n}(N^*),G,G].$$
 It follows that   $\mho_1 [(\mho_n(N^*)),G]\equiv 1 \pmod{[\mho_{n+1}(N^*),G]}$. Then by induction hypothesis, we have
 $$\mho_{n+1}([N^*,G])= \mho_1(\mho_{n}[N^*,G])=\mho_1 ([\mho_n(N^*),G])\equiv 1 \pmod{[\mho_{n+1}(N^*),G]}.$$
 This completes the proof for odd primes $p$. The proof for the case $p=2$ is similar.
\end{proof}
\section*{Acknowledgement}
The authors would like to thank the referee for the valuable comments and useful suggestions to improve the present paper.

This research was supported by a grant from Ferdowsi University of Mashhad; (No. MP90259MSH).

\end{document}